\documentclass[a4paper]{amsart}
\usepackage{amsmath}
\usepackage{amssymb}
\usepackage[all]{xy}
\usepackage{amsmath, amsfonts, amsthm , amssymb, amscd}
\usepackage{hyperref}
\usepackage{booktabs}
\usepackage{todonotes}
\usepackage{mathtools}
\usepackage{enumitem}
\usepackage{romannum}

\usepackage{tikz}
\usetikzlibrary{decorations.pathmorphing}
\newtheorem*{theorem*}{Theorem A}
\newtheorem*{theorem**}{Theorem B}

\newtheorem{theorem}{Theorem}[section]

\newtheorem{proposition}[theorem]{Proposition}
\newtheorem{lemma}[theorem]{Lemma}

\theoremstyle{definition}
\newtheorem{definition}[theorem]{Definition}
\newtheorem{example}[theorem]{Example}

\theoremstyle{remark}

\newtheorem{problem}[theorem]{Problem}

\title{Existence of a small cover over a 15-colorable simple 4-polytope}

\author{Djordje Barali\'{c}}
\address{ \scriptsize{Mathematical Institute SANU, Belgrade, Serbia }}
\email{djbaralic@mi.sanu.ac.rs}


\date{}

\begin{document}

\pagenumbering{arabic}

\begin{abstract}

The chromatic number for properly colouring the facets of a combinatorial simple $n$-polytope $P^n$ that is the orbit space of a quasitoric manifold satisfies the inequality $n\leq P^n\leq 2^n-1$. The inequality is sharp for $n=2$ but not for $n=3$ due to the Four Color theorem. In this note, we construct a simple 4-polytope admitting a characteristic map whose chromatic number equals $15$ and deduce that the predicted upper bound is attained for $n=4$. Analogues results are verified for the case of oriented small covers in dimensions $4$ and $5$.

\end{abstract}

\maketitle

\section{Introduction}

Toric topology emerged as a new mathematical discipline at the crossroads of algebraic topology, combinatorics, algebraic geometry, commutative algebra and symplectic geometry in the last decades of the 20th century. It studies spaces with toric actions whose orbit spaces have nice combinatorial structures such as simple polytopes and simplicial complexes. In their remarkable monograph \cite{BuPan1} Buchstaber and Panov summarized the main results obtained in this field until 2012 and gave an overview of the main open problems in this area.

In their seminar paper \cite{Davis} Davis and Januszkiewicz motivated the study of small covers and quasitoric manifolds. The orbit space of locally standard (real) torus action on a quasitoric manifold (small cover) is a combinatorial simple polytope, and its topology is closely related to the combinatorics of the corresponding combinatorial simple polytope and the characteristic matrix. Indeed, the classical problem of classification of quasitoric manifolds and small covers up to a weak equivariant diffeomorphism over a given combinatorial simple polytope $P$ is equivalent to the classification of all characteristic matrices over $P$ modulo the actions of $GL (n)$ and the group $\mathrm{Aut} (P)$ of all automorphism of $P$. Indeed, even the existence of a smooth structure on a quasitoric manifold is a non-trivial problem. The original construction, due to \cite{Davis}, was fully justified only around 2014 by Davis in \cite{Davis1}.The canonical one was defined in 2007 by Buchstaber, Panov, Ray in \cite{BPR} and is coming from the corresponding moment-angle manifold. In the last decades, many researchers worked on this and related questions and achieved significant results. The interested reader may find more about them in the survey paper \cite{Suh} by Masuda and Suh. However, despite the tremendous efforts, the fundamental question remained open.

The combinatorial characterization of simple $n$-polytopes admitting a characteristic map still needs to be discovered. Apart from the condition $n\leq \chi(P^n)\leq 2^n-1$ on the chromatic number of a simple polytope, established by Davis and Januszkiewicz, there was no significant progress in the general case. The two and three-dimensional cases are understood completely since every polygon and every simple 3-polytope is the orbit space of a small cover (quasitoric manifold). The celebrated Four Color theorem impacts the case $n=3$, for details see \cite{BuPan2}. The lower bound in the inequality is the best possible by \cite[Example~1.15]{Davis}, but up to the author's best knowledge, the sharpness of the upper bound still needs to be studied.

This article aims to popularize the interest in this question, as it is closely related to other problems in toric topology, such as the characterization of the Buchstaber invariant of a simple polytope and the lifting conjecture. In the paper, we give an algorithmic and combinatorial construction of a simple $4$-polytope admitting a characteristic map whose chromatic number is 15, the highest possible predicted by the general inequality. Section 2 briefly reviews the main facts about simple polytopes, small covers and quasitoric manifolds. The main result is formulated and proved in Section 4, where we present some applications of our method on the upper bound for the chromatic numbers of simple polytopes of dimensions 4 and 5 that are the orbit space of an oriented small cover. These results explain the limitation of our approach in higher dimensions.

\section{Preliminaries}

\subsection{Simple polytopes} In this section, we briefly present the basics of the theory of simple polytopes based on the classical monographs
\cite{Grunbaum} and \cite{Zieg}.

\textit{A convex polytope} is the convex hull of a
finite set of points in some $\mathbb{R}^n$. The \textit{dimension} of a polytope is the dimension of its
affine hull. Let us assume that a point $x\in \mathbb{R}^n$ is represented as
a column vector and that $\mathbf{a}\in
(\mathbb{R}^n)^\ast$ is a row vector. Then there exists a linear form $l_\mathbf{a} \colon
\mathbb{R}^n\rightarrow \mathbb{R}$ such that
$\mathbf{x}\mapsto \mathbf{ax}$. Moreover, any linear form $l \colon
\mathbb{R}^n\rightarrow \mathbb{R}$ has this form for some $\mathbf{a}\in (\mathbb{R}^n)^\ast$. If for all points $\mathbf{x}\in P\subseteq \mathbb{R}^n$ holds
inequality  $\mathbf{m} \mathbf{x}\leq r$, then we say that it \textit{valid} for a
convex polytope $P$. A \textit{face} of $P$ is any set of
the form
$$F=P\cap \left\{\mathbf{x}\in \mathbb{R}^n \left|\right. \mathbf{m} \mathbf{x}=
r\text{ and } \mathbf{m} \mathbf{x}\leq r \text{ a valid inequality}\right\}.$$ The dimension of a face is the dimension of its
affine hull.

As the obvious inequalities $\mathbf{0} \mathbf{x}\leq 0$ and
$\mathbf{0} \mathbf{x}\leq 1$ trivially hold, we see that the polytope $P$
itself and $\emptyset$ are faces of $P$. All other faces of $P$
are called \textit{proper} faces. The faces of dimension $0$ and $\dim (P)-1$ are called \textit{vertices} and \textit{facets}, respectively.
The faces of a polytope $P$ are polytopes of smaller dimensions, and
every intersection of a finite number of faces is a face of $P$,
\cite[Proposition~2.3]{Zieg}. All faces of a convex polytope $P$
form a poset $L:=L (P)$ \textit{face lattice} called the with respect to inclusion.

For two polytopes $P_1$ and $P_2$ having isomorphic face lattices $L
(P_1)$ and $L (P_2)$ we say they are
\textit{combinatorially equivalent}. A combinatorial polytope is a
class of combinatorially equivalent polytopes.

An \textit{ $n$-simplex} $\Delta^n$ is an $n$-dimensional polytope which is the convex hull of its $n+1$ vertices. We say that polytope $P$ is \textit{simplicial} if all its proper faces are simplices.

The \textit{polar set} $P^\ast \subset (\mathbb{R}^n)^\ast$ of a convex polytope $P\subset \mathbb{R}^n$ is defined by
$$P^\ast :=\left\{ \mathbf{c}\in (\mathbb{R}^n)^\ast \left|\right.
\mathbf{c x}\leq 1 \mbox{\, for all\,} \mathbf{x}\in P\right\}.$$
We refer to the combinatorial polytope $P^\ast$ as \textit{the
dual of the combinatorial polytope $P$}. The face lattice $L
(P^\ast)$ is the opposite of the face lattice $L (P)$ of $P$.

 A polytope $P$ is called \textit{simple} if its dual polytope $P^\ast$ is simplicial. Each face of a simple
polytope is a simple polytope. A simple polytope $P^n$ is characterized by the property that each vertex belongs to exactly $n$ facets. Any combinatorially polar polytope of a simple polytope is simplicial.

\subsection{Neighborly polytopes}

Let $P$ be an $n$-polytope. We say that $P$ is  \textit{$k$-neighborly} if any
subset of $k$ or fewer vertices is the vertex set of a face of $P$.
It can be verified that the simplex $\Delta^n$ is the only $k$-neighborly polytope for $k> \left
[\frac{n}{2}\right ]$. Thus, polytopes that are
$\left[\frac{n}{2}\right ]$-neighborly are of particular interests
and are called \textit{neighborly} polytopes.

The neighborly polytopes often appear in
combinatorics as solutions of various extremal
problems. They satisfy the upper bound predicted by Motzkin for
a maximal number of $i$-faces of an $n$-polytope with $m$ vertices.
This statement is known as the Upper Bound Theorem, which was
first proved by McMullen in \cite{Mullen1}.

The most known example of a neighborly $n$-polytope with $m$ vertices
is \textit{the cyclic polytope $C^n (m)$}. The
cyclic polytope $C^n (m)$ is the convex hull of $m$ points  on \textit{the
moment curve} $\gamma$ in $\mathbb{R}^n$,  defined by $\gamma :
\mathbb{R}\rightarrow \mathbb{R}^n$, $t\mapsto
\mathbf{\gamma}(t)=(t, t^2, \dots, t^n)\in \mathbb{R}^n$. The combinatorial class of $C^n (m)$ does not depend
on the specific choices of the points due to Gale's
evenness condition, see \cite[Theorem~0.7.]{Zieg}. Gale's result also implies that
cyclic polytopes are simplicial polytopes.

The combinatorial neighborly $n$-polytope with less than $n+4$ vertices is cyclic. However, for a fixed $n$, the number
of combinatorially different neighborly polytopes grows
superexponentially with the number of vertices $m$. This result follows from Barnette's construction from \cite{Barnette}  and Shemer from \cite{Shemer}. Despite the vast number of combinatorially different neighborly polytopes apart from cyclic polytopes, it is unknown if some `nice description characterizes another subclass.

Duals of simplicial neighborly $n$-polytopes are simple polytopes
with property that each $\left[ \frac{n}{2}\right]$ facets have
nonempty intersections. Such polytopes are also called
\textit{neighborly} polytopes, and in the rest of the paper, we use the term
neighborly polytopes for simple neighborly polytope.

\subsection{Small covers and quasitoric manifolds}

 Quasitoric manifolds and small covers appeared as topological generalizations of smooth projective  (real) toric varieties in the seminal paper \cite{Davis}. Since then, they have been extensively studied in
 toric topology. A detailed exposition on
 the subject can be found in Buchstaber and Panov's monographs
\cite{BuPan} and \cite{BuPan1}. Here we briefly review the primary
definition and results about them.

 Let us denote by $\mathbb{Z}_2=\{-1, 1\}$ and $S^1=\{z \left|\right. |z|=1\}$ are
multiplicative subgroups of real and complex numbers,
respectively. We denote $(S^1)^n$  by $T^n$. The group $\mathbb{Z}_2^n$ acts on $\mathbb{R}^n$ by coordinatewise multiplication. Analogously, the group $T^n$ acts on $\mathbb{C}^n$ and both actions are called \textit{the standard actions}.
 A \textit{$\mathbb{Z}_2^n$ ($T^n$)-manifold} is a differentiable manifold
with a smooth action of $\mathbb{Z}_2^n$ ($T^n$).

\begin{definition} A map $f\colon X\rightarrow Y$ between two $\mathbb{Z}_2^n$ ($T^n$)-spaces
$X$ and $Y$ is called weakly equivariant if for any $x\in X$ and
$g\in \mathbb{Z}_2^n$ ($T^n$) holds $$f (g\cdot x)=\psi (g) \cdot f(x),$$  where $\psi
\colon \mathbb{Z}_2^n\rightarrow \mathbb{Z}_2^n$ ($\psi
\colon T^n\rightarrow T^n$) is some automorphism of group  $\mathbb{Z}_2^n$ ($T^n$).
\end{definition}

Let $M$ be a $n$($2n$)-dimensional $\mathbb{Z}_2^n$ ($T^n$)-manifold. A
\textit{standard chart} on $M$ is an ordered pair $(U, f)$,
where $U$ is a $\mathbb{Z}_2^n$ ($T^n$)-invariant open subset of $M$ and $f$ is a
weak equivariant diffeomorphism from $U$ onto some
$\mathbb{Z}_2^n$ ($T^n$)-invariant open subset of $\mathbb{R}^n$ ($\mathbb{C}^n$). A \textit{standard
atlas} is an atlas which consists of standard charts. A $\mathbb{Z}_2^n$ ($T^n$)
action on a $\mathbb{Z}_2^n$ ($T^n$)-manifold $M$ is called locally standard
if manifold $M$ has a standard atlas. The orbit space for a
locally standard action is naturally regarded as a manifold with
corners.

\begin{definition}\label{gdm} A \textit{small cover (quasitoric manifold)} $\pi \colon M \rightarrow P^n$  is a smooth closed $n$ ($2n$)-dimensional $\mathbb{Z}_2^n$ ($T^n$)-manifold
admitting a locally standard ${\mathbb{Z}_2}^n$ ($T^n$)-action such that its orbit
space is a simple convex $n$-polytope $P^n$ regarded as a manifold
with corners.
\end{definition}

It is straightforward to check that two weakly equivariant small covers (quasitoric manifolds) have combinatorially equivalent orbit polytopes so
we do not have to distinguish combinatorially equivalent  simple polytopes in the above
definition.

Let $P^n$ be a simple polytope with $m$ facets $F_1$, $\dots$,
$F_m$. By Definition \ref{gdm}, it follows that every point in
$\pi^{-1} (\mathrm{rel.int} (F_i))$ has the same isotropy group
which is a one-dimensional subgroup of $\mathbb{Z}_2^n$ ($T^n$). We denote it by
$\mathbb{Z}_2 (F_i)$ ($S^1 (F_i)$).  Each small cover (quasitoric manifold) $\pi: M \rightarrow
P^n$ determines a \textit{characteristic map} $l_d$ on $P^n$
$$l_d : \{F_1, \dots, F_m\}\rightarrow{\mathbb{Z}_2}^n ({\mathbb{Z}}^n)$$ defined by mapping each facet of
$P^n$ to nonzero elements of ${\mathbb{Z}_2}^n ({\mathbb{Z}}^n)$ such that \\ $l_d (F_i)=
\boldsymbol{\lambda_i} = (\lambda_{1, i}, \dots, \lambda_{n, i})^t
\in {\mathbb{Z}_2}^n ({\mathbb{Z}}^n)$, where $\mathbf{\lambda}_i$ is a primitive integral
vector such that $$\mathbb{Z}_2 (F_i)=\left\{(t^{\lambda_{1, i}}, \dots,
t^{\lambda_{n, i}})| t\in \mathbb{R}, |t|=1\right\} (S^1 (F_i)=\left\{(t^{\lambda_{1, i}}, \dots,
t^{\lambda_{n, i}})| t\in \mathbb{C}, |t|=1\right\}).$$ From the
characteristic map we obtain an integer $(n\times m)$-matrix
$\Lambda_{\mathbb{Z}_2} (M):=(\lambda_{i, j})$ ($\Lambda_{\mathbb{Z}} (M):=(\lambda_{i, j})$) which is called
\textit{the characteristic matrix of $M$}. For the quasitoric manifold each
$\boldsymbol{\lambda}_i$ is determined up to a sign. Since the
$\mathbb{Z}_2^n$ ($T^n$)-action on $M$ is locally standard, the characteristic
matrix $\Lambda_{\mathbb{Z}_2} (M)$ ($\Lambda_{\mathbb{Z}} (M)$ ) satisfies the non-singularity
condition for $P^n$, i.e. if $n$ facets $F_{i_1}$, $\dots$,
$F_{i_n}$ of $P^n$ meet at a vertex, then $\left |\det
\Lambda_{\mathbb{Z}_2}^{(i_1, \dots, i_n)} (M) \right |=1$ ($\left |\det
\Lambda_{\mathbb{Z}_2}^{(i_1, \dots, i_n)} (M) \right |=1$), where
$\Lambda_{\mathbb{Z}_2}^{(i_1, \dots, i_n)}
(M):=(\boldsymbol{\lambda}_{i_1}, \dots,
\boldsymbol{\lambda}_{i_n})$($\Lambda_{\mathbb{Z}_2}^{(i_1, \dots, i_n)}
(M):=(\boldsymbol{\lambda}_{i_1}, \dots,
\boldsymbol{\lambda}_{i_n})$). Any integer $(n\times m)$-matrix
satisfying the non-singularity condition for $P^n$ is also called
\textit{a characteristic matrix on $P^n$}.

The construction of a small cover and a quasitoric manifold from
the {\it characteristic pair} $(P^n, \Lambda_{\mathbb{Z}_2})$ ($(P^n, \Lambda_{\mathbb{Z}})$) where
$\Lambda_{\mathbb{Z}_2}$ ($\Lambda_{\mathbb{Z}}$) is a characteristic matrix is described in \cite[Construction~1.5]{Davis},
\cite{BuPan1} and \cite[Construction~5.12]{BuPan}. For each
point $x\in P^n$, we denote the minimal face containing $x$ in its
relative interior by $F(x)$. The characteristic map $l$
corresponding to $\Lambda_{\mathbb{Z}_2}$ ($\Lambda_{\mathbb{Z}}$) is a map from the set of the
faces of $P^n$ to the set of subtori of $\mathbb{Z}_2^n$ ($T^n$) defined by
$$l (F_{i_1}\cap \dots \cap F_{i_k}):=
\mathrm{im}\left(l^{(i_1, \dots, i_k)} \colon \mathbb{Z}_2^k\, (T^k)\rightarrow
\mathbb{Z}_2^n\,(T^n)\right ),$$ where $l^{(i_1, \dots, i_k)}$ is the map
induced from the linear map determined by $\Lambda_{\mathbb{Z}_2}^{(i_1,
\dots, i_k)}$ ($\Lambda_{\mathbb{Z}}^{(i_1,
\dots, i_k)}$). A $\mathbb{Z}_2^n$ ($T^n$)-manifold $M$ over
simple polytope $P^n$ is obtained by setting $$M
(\Lambda_{\mathbb{Z}_2}):=(\mathbb{Z}_2^n\times P^n)/\sim_{l}\, (M
(\Lambda_{\mathbb{Z}}):=(T^n\times P^n)/\sim_{l}),$$ where
$\sim_{l}$ is an equivalence relation defined by $(t_1,
p)\sim_{l} (t_2, q)$ if and only if $p=q$ and $t_1 t_2^{-1}\in
l_d (F(q))$. The free action of $\mathbb{Z}_2^n$ ($T^n$) on $\mathbb{Z}_2^n\times P^n$ ($T^n\times P^n$)
obviously descends to an action on $(\mathbb{Z}_2^n\times P^n)/\sim_{l}$ ($(T^n\times P^n)/\sim_{l}$)
with quotient $P^n$. Simple polytope $P^n$ is covered by the open
sets $U_v$ obtained by deleting all faces not containing vertex
$v$ of $P^n$. Clearly, $U_v$ is diffeomorphic to $\mathbb{R}_+^n$ ($\mathbb{C}_+^n$),
so the space $(\mathbb{Z}_2^n\times P^n)/\sim_{l_d}$ ($(T^n\times P^n)/\sim_{l}$) is covered by open
sets $(\mathbb{Z}_2^n\times U_v)/\sim_{l_d}$ ($(T^n\times U_v)/\sim_{l}$) homeomorphic to
$\mathbb{R}^n$ ($\mathbb{C}^n$). We easily see that the transition maps are
diffeomorphisms, so $\mathbb{Z}_2^n$ ($T^n$)-action on $(\mathbb{Z}_2^n\times P^n)/\sim_{l}$ ($(T^n\times P^n)/\sim_{l}$)
is locally standard and $M (\Lambda_{\mathbb{Z}_2})$ ($M (\Lambda_{\mathbb{Z}})$ ) is a
small cover (quasitoric manifold) $\pi \colon M \rightarrow P^n$ over simple
convex $n$-polytope $P^n$.

The classification problem of small covers (quasitoric manifolds) over a given combinatorial simple polytope $P$ is closely related to determining the characteristic matrices over $P$. There are various types of classification: up to a homeomorphism, up to a diffeomorphism, up to a weakly equivariant diffeomorphic, etc. In the literature, the authors usually assume the following technical lemma.

\begin{lemma} \label{l1} Let $M$ be a small cover (quasitoric manifold) over $P^n$ with characteristic
map $l$. There exists a weak equivariant diffeomorphism
$f: M\rightarrow M$ induced by an automorphism $\psi$ of
$\mathbb{Z}_2^n$  ($T^n$) such that the characteristic matrix induced by $f$ has the
form $\left(I_{n\times n}| \ast\right)$ where $\ast$ denotes some
$n\times (m-n)$ matrix.
\end{lemma}

\subsection{Chromatic number as an obstruction to the existence of a characteristic map}

The principal question of toric topology is to classify small covers and quasitoric manifolds over a given combinatorial simple polytope $P^n$. From the above discussion, we know this problem is related to classifying the characteristic maps over $P^n$. Indeed, by \cite[Nonexamples~1.23]{Davis}, we know that a $2$-neighborly simple polytope with
$m\geq 2^n$ does not admit a characteristic map. One of the main open problems in toric topology is to find a combinatorial description of the class of
polytopes $P^n$ admitting a characteristic map.

Rarely obstructions to the existence of a characteristic map are known. The most studied are those related to \textit{the chromatic number} of $P^n$.
Recall that the chromatic number of $P^n$ is the minimal integer $k$ such that there is a proper coloring of its facets into $k$ colors. Under proper coloring, we assume that no two distinct facets sharing a vertex have the same color.

Obviously, $\chi (P^n)\geq n$ for any simple polytope $P^n$. The
chromatic number of a $2$-dimensional simple polytope is clearly
equal to $2$ or $3$, depending on the parity of the number of its
facets. By the famous Four Color Theorem, we deduce that the
chromatic number of a $3$-dimensional polytope is $3$ or $4$. Simple polytopes $P^n$ such that $\chi (P^n)= n$ or $\chi (P^n)= n+1$ are the orbit spaces of small covers and quasitoric manifolds, see \cite[Example~1.15]{Davis}. The class of simple polytopes admitting a characteristic map contains
some important combinatorial simple polytopes, such as the simplex,
the cube, the permutahedron, polygons, $3$-dimensional polytopes
etc.

However, for $n\geq 4$, we do not have an analogue of the Four Color Theorem, so $\chi (P^n)$  can be arbitrarily high. A typical example is a dual of the cyclic polytope $C^n(m)$ whose chromatic number is $m$. An immediate obstruction at the chromatic number comes from the existence of a characteristic map. If a simple polytope $P^n$ is the orbit space of a small cover or quasitoric manifold, we have  at
most $2^n-1$ different possibilities for $\boldsymbol{\lambda}_i$
modulo $2$. Therefore, a coloring with no
more than $2^n-1$ colors arises from  a characteristic map, so it holds that
\begin{equation}\label{nej}
\chi (P^n)\leq 2^n-1.
\end{equation}

For $n=2$, the inequality \ref{nej} is sharp, while for $n=3$, it does not coincide with the bound implied by Four Color Theorem. In the next section, we show that the inequality is sharp for $n=4$ by providing an explicit example of a simple $4$-polytope admitting a characteristic map such that its chromatic number is $15$.

\section{Proof of the main theorem}

The following theorem is our main contribution.

\begin{theorem}\label{main} There exists a combinatorial simple $4$-polytope $P$ such that $\chi (P)=15$ that is the orbit space of a small cover and quasitoric manifold.
\end{theorem}

\begin{proof} Let us consider a neighborly $4$-polytope with $15$ facets. For example, we can take the dual of $C^4 (15)$. Let us map its facets surjectively onto the set of nonzero vectors of $\mathbb{Z}_2^4$. Let us denote by $v_F$ the image of a facet $F$ by the surjection. We are done if the non-singularity condition is fulfilled in all vertices, as the surjection defines the characteristic map we need.

Otherwise, the non-singularity condition may be broken in two ways.

The first one is on an edge of the polytope obtained as the intersection of three facets $F_1, F_2$ and $F_3$ such that $v_{F_1}+v_{F_2}+v_{F_3}=0$. We call such an edge bad. Another possibility is that the non-singularity condition fails in a vertex of the polytope that it is the intersection of four facets $F_1, F_2, F_3$ and $F_4$ such that $v_{F_1}+v_{F_2}+v_{F_3}+v_{F_4}=0$. We say that such a vertex is bad.

Now let us `resolve' the problem in a bad edge. Without lose of generality we may assume that $v_{F_1}=e_1, v_{F_2}=e_2$ and $v_{F_3}=e_1+e_2$. The endpoints of the bad edge are the vertices $F_1\cap F_2\cap F_3\cap F_4$ and $F_1\cap F_2\cap F_3\cap F_5$ where $F_4$ and $F_5$ are some facets of the polytope such that $v_{F_4}=a e_1+b e_2+ c e_3 +d e_4$ and $v_{F_5}=p e_1+q e_2+ r e_3 +s e_4$ where $e_i$, $i=1,\dots, 4$ is the standard bases and $a, b, c, d, p, q, r,s\in \{0, 1\}$ are integers such that $c+d>0$ and $r+s>0$. We truncate the polytope over the bad edge, and a new facet $F'$ arises from this operation. $F'$ is a triangular prism with vertices $F_1\cap F_2\cap F_4 \cap F'$, $F_1\cap F_3\cap F_4 \cap F'$, $F_2\cap F_3\cap F_4 \cap F'$, $F_1\cap F_2\cap F_5 \cap F'$, $F_1\cap F_3\cap F_5 \cap F'$ and $F_2\cap F_3\cap F_5 \cap F'$. We have that $\{c e_3+d e_4, r e_3+s e_4\}\subset \{e_3, e_4, e_3+e_4\}$ so there is a vector $v_{F'}\in\{e_3, e_4, e_3+e_4\}$ such that $v_{F'}\neq c e_3+d e_4$ and $v_{F'}\neq r e_3+s e_4$. Now, we directly verify that with $v_{F'}$, the non-singularity condition is satisfied in the new six vertices. The chromatic number of newly obtained polytope remains $15$; in the new polytope $F_1\cap F_2\cap F_3=\emptyset$, but pairwisely the intersections of $F_1$, $F_2$ and $F_3$ are nonempty.

Moreover, applying this algorithm successfully on the bad edges, we can get rid of them to get a simple $4$-polytope $Q$ with the chromatic number $15$ containing no bad edge. Now let us show how to `resolve' singularities in bad vertices. For a bad vertex $F_1\cap F_2\cap F_3\cap F_4$ without lose of generality we may assume that $v_{F_1}=e_1$, $v_{F_2}=e_2$, $v_{F_3}=e_3$ and $v_{F_4}=e_1+e_2+e_3$. We truncate $Q$ over the bad vertex to obtain a new facet $F'$. The vertices of the tetrahedra $F'$ are $F_1\cap F_2\cap F_3 \cap F'$, $F_1\cap F_2\cap F_4 \cap F'$, $F_1\cap F_3\cap F_4 \cap F'$ and $F_2\cap F_3\cap F_4 \cap F'$. The non-singularity condition we can obey by setting $v_{F'}=e_4$. Again, it is clear that we have kept the chromatic number the same, so we can continue truncation until we get a polytope $P$ having neither a bad vertex nor an edge. The vectors assigned to the facets of $P$ define the desired characteristic function.

\end{proof}

The construction above cannot be easily generalized to dimensions greater than $4$ due to the possible existence of a bad 2-face that can be a polygon with many sides. However, the product of the polytope $P$ from Theorem \ref{main} and the segment $I$ gives a combinatorially simple $5$-polytope that admits a characteristic map and whose chromatic number is $16$. Although this is still far away from the predicted upper bound $31$, it is greater than $7$, that is the chromatic number of the product of a simple $3$-polytope $P_1$ and a polygon $P_2$ such that $\chi (P_1)=4$ and $\chi (P_2)=3$. Recall that by \cite[Subsection~1.10]{Davis}, we know that if $M$ and $N$ are small covers (quasitoric manifolds) over simple polytopes $P^m$ and $Q^n$  then $M \times N$ is a small cover (quasitoric manifold) over the polytope $P^m\times Q^n$. It is intractable to answer the following question.

\begin{problem} Does there exist a combinatorial simple $n$-polytope $P$ admitting a characteristic map whose chromatic number is $2^n-1$ for $n\geq 5$?
\end{problem}

Quasitoric manifolds are always oriented, but small covers may fail to be. Nakayama and Nishimura found an orientability condition for a small cover \cite{Nakay}. Their criterium states that a small cover $M$ over $P^n$ is oriented if and only if there is a map $l$ that maps the facets of $P^n$ into nonzero vectors of $\mathbb{Z_2}^n$ such that for every face $F$ of $P^n$ $l (F)$ has an odd number of nonzero coordinates. Therefore, the following analogue of the inequality \ref{nej} holds.

\begin{proposition} If a simple $n$-polytope $P$ is the orbit space of an oriented small cover then
\begin{equation}\label{ne1}
\chi (P^n)\leq 2^{n-1}.
\end{equation}
\end{proposition}

The inequality is known to be sharp for $n=2$ and $n=3$ by \cite{Nakay}. We can apply the same argument as in the proof of Theorem \ref{main} to show the following claim.

\begin{theorem}\label{main2} There exists a combinatorial simple $4$-polytope $P$ such that $\chi (P)=8$ that is the orbit space of an oriented small cover.
\end{theorem}

Indeed, the argument can be applied to the case of oriented $5$-dimensional small covers.

\begin{theorem}\label{main3} There exists a combinatorial simple $5$-polytope $P$ such that $\chi (P)=16$ that is the orbit space of an oriented small cover.
\end{theorem}

\begin{proof} We use the same approach as in the proof of Theorem \ref{main}. We start with the dual of $C^4 (15)$ and decorate its facets by nonzero vectors of $\mathbb{Z_2}^5$ having an odd number of nonzero coordinates. Again the non-singularity condition may fail to be satisfied. Fortunately, there is no bad 2-face arising from the distribution of the vectors across the facets. If there is a bad $3$-polytope $F_1\cap F_2\cap F_3$, then the vector $v_{F_1}+v_{F_2}+v_{F_3}$ certainly has an odd number of nonzero coordinates so it cannot be the zero vector. Therefore, the non-singularity condition may fail only in bad vertices or edges. However, we can efficiently resolve such problems by truncation and procedure we already explained with appropriate modifications for the vector $v_{F'}$ associated with the new facet $F'$ coming from a bad edge.
\end{proof}

In the end, we give an example of a combinatorial $4$-polytope obtained by the algorithm described in the proof of Theorem \ref{main}.

\begin{example} Let us decorate the facets of the dual of the cyclic polytope $C^4 (15)$ in the following way\\
$ F_0\mapsto e_1,  F_1\mapsto e_1+e_2, F_2\mapsto e_3,  F_3\mapsto e_4,
 F_4\mapsto e_1+e_4,  F_5\mapsto e_1+e_2+e_4,  F_6\mapsto e_2+e_4, F_7\mapsto e_2,
 F_8\mapsto e_2+e_3,  F_{9}\mapsto e_1+e_2+e_3,  F_{10}\mapsto e_1+e_3,  F_{11}\mapsto e_1+e_3+e_4,
 F_{12}\mapsto e_3+e_4,  F_{13}\mapsto e_2+e_3+e_4,  F_{14}\mapsto e_1+e_2+e_3+e_4$,
where $e_1, e_2, e_3, e_4$ are the vector of the standard basis of $\mathbb{Z}_2^4$. By Gale's
evenness condition (see \cite[Theorem~3.]{Gale}), the vertices of this polytope are $F_0\cap F_x\cap F_{x+1}\cap F_{14}$ for all $1\leq x\leq 12$ and $F_x\cap F_{x+1}\cap F_y\cap F_{y+1}$ for all $x, y\in\{0, \dots, 13\}$ such that $y-x\geq 2$.

In such decoration, there are $13$ `bad' edges:\\
$ F_2\cap F_7\cap F_{8}, F_1\cap F_2\cap F_9, F_3\cap F_{10}\cap F_{11}, F_2\cap F_3\cap F_{12},
 F_4\cap F_5\cap F_{7}, F_7\cap F_9\cap F_{10}, F_7\cap F_{12}\cap F_{13}, F_0\cap F_3\cap F_{4},
 F_0\cap F_5\cap F_{6}, F_0\cap F_8\cap F_9, F_0\cap F_{11}\cap F_{12}, F_0\cap F_{13}\cap F_{14},
 F_0\cap F_1\cap F_{7}$,\\
and $17$ `bad' vertices:\\
$F_3\cap F_4\cap F_{5}\cap F_{6}, F_3\cap F_4\cap F_{8}\cap F_{9}, F_3\cap F_4\cap F_{11}\cap F_{12},
 F_3\cap F_4\cap F_{13}\cap F_{14}, F_4\cap F_5\cap F_{9}\cap F_{10}, F_4\cap F_5\cap F_{12}\cap F_{13},
 F_5\cap F_6\cap F_{8}\cap F_{9}, F_5\cap F_6\cap F_{11}\cap F_{12}, F_5\cap F_6\cap F_{13}\cap F_{14},
 F_6\cap F_7\cap F_{10}\cap F_{11}, F_8\cap F_9\cap F_{11}\cap F_{12}, F_8\cap F_9\cap F_{13}\cap F_{14},
 F_9\cap F_{10}\cap F_{12}\cap F_{13}, F_{11}\cap F_{12}\cap F_{13}\cap F_{14}, F_0\cap F_1\cap F_{4}\cap F_{5},
 F_0\cap F_1\cap F_{9}\cap F_{10}, F_0\cap F_1\cap F_{12}\cap F_{13}$.

Taking the truncations over the bad edges and vertices we produce a new polytope $P$ with 193 vertices, 386 edges, 228 ridges and 45 facets. Its face lattice can be easily explicitly obtained.
\end{example}

Small covers over the duals of cyclic polytopes do not exist when the difference between the number of facets and the dimension is greater than $4$ by \cite{Hasui}. However, the example of a simple neighborly $4$-polytope with 12, which is the orbit space of a small cover, was found by Barali\'{c} and Milenkovi\'{c} in \cite{BarMil}. We are curious if there is a simple neighborly $4$-polytope with 15 facets admitting a characteristic map.

\section*{Acknowledgements}

The author was supported by the Serbian Ministry of Science, Innovations and Technological Development through the Mathematical Institute of the Serbian Academy of Sciences and Arts. The authors are grateful to Ivan Limonchenko for the discussion and valuable comments.

\end{document}